\documentclass[a4paper,pdftex,reqno,10pt]{amsart}%

\usepackage{amssymb,amsmath}
\usepackage{mathptmx}       % selects Times Roman as basic font
\usepackage{helvet}         % selects Helvetica as sans-serif font
\usepackage{courier}        % selects Courier as typewriter font
\usepackage{type1cm}        % activate if the above 3 fonts are
\usepackage{url}
\usepackage{algorithm2e}
\usepackage{graphicx}        % standard LaTeX graphics tool
\usepackage{multicol}        % used for the two-column index
\usepackage[bottom]{footmisc}% places footnotes at page bottom

\usepackage{bm} % for bold math fonts
\usepackage{enumerate} % Customizable enumerated lists
\usepackage{ytableau} % for drawing echelon-Ferrers diagrams

\newcommand{\F}{\mathbb{F}}
\newcommand{\gauss}[3]{\genfrac{[}{]}{0pt}{}{#1}{#2}_{#3}}
\newcommand{\SPG}[2]{\operatorname{PG}(#1,#2)}
\newcommand{\PG}{\operatorname{PG}}

\newcommand{\aspace}{\SPG{v-1}{\mathbb{F}_q}}

\newcommand{\smax}{\mathrm{A}}

\newcommand{\gaussmnum}[3]{\left[\begin{smallmatrix}{#1}\\{#2}\end{smallmatrix}\right]_{#3}}
\newcommand{\gaussmset}[2]{\left[\begin{smallmatrix}{#1}\\{#2}\end{smallmatrix}\right]}

\newtheorem{theorem}{Theorem}
\newtheorem{proposition}{Proposition}
\newtheorem{lemma}{Lemma}
\newtheorem{corollary}{Corollary}

\newtheorem{remark}{Remark}

\makeindex             % used for the subject index
\begin{document}

\title{Johnson type bounds for mixed dimension subspace codes}

\author{Thomas Honold, Michael Kiermaier, and Sascha Kurz}
\address{Thomas Honold, Zhejiang University, 310027 Hangzhou, China}
\email{honold@zju.edu.cn}
\address{Michael Kiermaier,University of Bayreuth, 95440 Bayreuth, Germany}
\email{michael.kiermaier@uni-bayreuth.de}
\address{Sascha Kurz, University of Bayreuth, 95440 Bayreuth, Germany}
\email{sascha.kurz@uni-bayreuth.de}

\abstract{Subspace codes, i.e., sets of subspaces of $\mathbb{F}_q^v$, are applied in random linear network coding. 
Here we give improved upper bounds for their cardinalities based on the Johnson bound for constant dimension codes.\\[2mm]
\textbf{Keywords:} Galois geometry, network coding, subspace code, partial spread.\\
\textbf{MSC:} Primary  05B25,  51E20;  Secondary  51E14, 51E23.
}}
\maketitle

\section{Introduction}
Driven by the application in random linear network coding, the field of subspace coding received a lot of interest recently.
Various upper bounds on the size of a subspace code have been developed.
In the special case of the constant dimension codes, the Johnson bound stands out as in many cases it leads to the sharpest known bounds.
In this article we will investigate the Johnson bound for applicability in the case of general mixed dimension subspace codes.

Let $\mathbb{F}_q$ be the finite field with $q$ elements, where $q>1$
is a prime power. For $v\geq 1$ we denote by $\mathbb{F}_q^v$ the
$v$-dimensional standard vector space over $\mathbb{F}_q$. The set of
all subspaces of $\mathbb{F}_q^v$, ordered by the incidence relation
$\subseteq$, is called \emph{($v-1$)-dimensional projective geometry
  over $\mathbb{F}_q$} and denoted by $\aspace$ or $\PG(\F_q^v)$. It
forms a finite modular geometric lattice with meet
$X\wedge Y=X\cap Y$, join $X\vee Y=X+Y$, and rank function
$X\mapsto\dim(X)$. We will use the term \emph{$k$-subspace} to denote
a $k$-dimensional subspace of $\mathbb{F}_q^v$. Using geometric
terminology we also speak of points, lines, planes, and solids for
$1$-, $2$-, $3$-, and $4$-subspaces, respectively; $(v-1)$-subspaces are
also called hyperplanes. The set of all $k$-subspaces of
$V=\mathbb{F}_q^v$ will be denoted by $\gaussmset{V}{k}$. Its
cardinality is given by the Gaussian binomial coefficient
$$
\gauss{v}{k}{q} :=
\begin{cases}
	\frac{(q^v-1)(q^{v-1}-1)\cdots(q^{v-k+1}-1)}{(q^k-1)(q^{k-1}-1)\cdots(q-1)} & \text{if }0\leq k\leq v\text{;}\\
	0 & \text{otherwise.}
\end{cases}
$$

For applications in network coding the relevant metric is given by the
\emph{subspace distance}
$ d_S(X,Y):=\dim(X+Y)-\dim(X\cap Y)=2\cdot\dim(X+Y)-\dim(X)-\dim(Y)$,
which can also be seen as the graph-theoretic distance in the Hasse
diagram of $\aspace$.  A set
$\mathcal{C}$ of subspaces of $\F_q^v$ is called a \emph{subspace 
code}.  The \emph{minimum (subspace)  distance} of $\mathcal{C}$ is given by
$d = \min\{d_S(X,Y) \,:\, X,Y\in\mathcal{C}, X \neq Y\}$.  If all
elements of $\mathcal{C}$ have the same dimension, we call
$\mathcal{C}$ a \emph{constant dimension code}. By $\smax_q(v,d)$ we denote the maximum possible cardinality of a 
subspace code in $\F_q^v$ with minimum distance at least $d$. Analogously, $\smax_q(v,d;k)$ denotes the maximum cardinality 
of a constant dimension code with codewords of dimension $k$. 
Like in the classical case of codes in the Hamming metric, the determination of the exact 
value or bounds for $\smax_q(v,d)$ and $\smax_q(v,d;k)$ is an important problem. In this paper we will present some 
improved upper bounds. For a broader background we refer to \cite{etzionsurvey,COSTbook} and for the latest 
numerical bounds to the online tables at \url{http://subspacecodes.uni-bayreuth.de} \cite{TableSubspacecodes}.

Constant dimension codes with $d=2k$ are called \emph{partial $k$-spreads}. A vector space partition $\mathcal{P}$ of 
$\mathbb{F}_q^v$ is a set of nonzero subspaces such that each point of $\mathbb{F}_q^v$ is contained in exactly one element of $\mathcal{P}$. 
If $\mathcal{P}$ consists of $m_i$ subspaces of dimension $i$ for $1\le i\le v$, then we say that $\mathcal{P}$ has \emph{type} $1^{m_1} 2^{m_2} \dots v^{m_v}$. 

The remaining part of this paper is organized as follows. In Section~\ref{section_known_bounds} we review known 
upper bounds for subspace codes. Johnson type bounds for constant dimension codes are presented in Section~\ref{section_johnson_cdc} 
before the underlying concept is generalized to the mixed dimension case in Section~\ref{section_johnson_mdc}. 
Analytic upper bounds for $A_q(v,v-4)$ and $A_q(8,3)$ are then determined in Section~\ref{sec_analytical_results}. We draw a brief 
conclusion in Section~\ref{sec_conclusion}.

\section{Known upper bounds for mixed dimension codes}
\label{section_known_bounds}
As mentioned in the introduction, our main focus is on the
determination of $A_q(v,d)$. In that process we will often encounter
codes with a restricted set $K\subseteq \{0,1,\dots,v\}$ of possible
dimensions, so that we will also consider upper bounds for
$A_q(v,d;K)$, the maximum cardinality of a subspace code in
$\mathbb{F}_q^v$ with minimum distance at least $d$ and all codeword
dimensions contained in $K$. Especially,
$A_q(v,d;\{k\})=A_q(v,d;k)$. The most obvious facts about the numbers
$A_q(v,d;K)$ are summarized in \cite[Lemma
2.3]{honold2015constructions}: Clearly we have monotonicity in $d$ and
$K$, i.e., $A_q(v,d;K)\ge A_q(v,d';K)$ for $1\le d\le d'\le v$ and
$A_q(v,d;K)\le A_q(v,d;K')$ for
$K\subseteq K'\subseteq \{0,\dots,v\}$. By decomposing codes we obtain
$A_q(v,d;K\cup K')\le A_q(v,d;K)+A_q(v,d;K')$. Considering the dual
subspace code $\mathcal{C}^\perp=\{X^\perp;X\in\mathcal{C}\}$ of
$\mathcal{C}$ gives $A_q(v,d,K)=A_q(v,d,v-K)$ with
$v-K=\{v-k\,:\,k\in K\}$. Subspace distance $d=1$ permits to choose
all subspaces, i.e., $A_q(v,1;K)=\sum_{i\in K} \gaussmnum{v}{i}{q}$.
For subspace distance $d=2$ the optimal codes in the unrestricted
mixed dimension case are classified in \cite[Theorem
3.4]{honold2015constructions} with cardinalities
\begin{equation*}
  A_q(v,2)=\sum\limits_{i\equiv \lfloor v/2\rfloor\pmod 2} \gaussmnum{v}{i}{q}
\end{equation*}
if $v$ is even and 
\begin{equation*}
  A_q(v,2)=\sum\limits_{i\equiv 0\pmod 2} \gaussmnum{v}{i}{q}=\sum\limits_{i\equiv 1\pmod 2} \gaussmnum{v}{i}{q}
\end{equation*}
if $v$ is odd.  For subspace distance $d=v$ we have $A_q(v,v)=2$ if
$v$ is odd and $A_q(v,v)=A_q(v,v;k)=q^k+1$ if $v$ is even, see
\cite[Theorem 3.1]{honold2015constructions}. Also subspace distance
$v-1$ has been resolved completely, see \cite[Theorem
3.2]{honold2015constructions}.  If $v=2m$ is even then
$A_q(v,v-1)=A_q(v,v;m)=q^m+1$, and if $v=2m+1\ge 5$ is odd then
$A_q(v,v-1)=A_q(v,v-1;m)=q^{m+1}+1$. Note that $A_q(3,2)=q^2+q+2>q^2+1$. So, in
the following we can always assume $v\ge 5$ and $3\le d\le v-2$. For
subspace distance $d=v-2$ there is so far only partial information,
see \cite[Theorem 3.3]{honold2015constructions}. If $v=2m\ge 8$ is
even, then $A_q(v,v-2)=A_q(v,v-2;m)$ with
$q^{2m}+1\le A_q(v,v-2;m)\le (q^m+1)^2$. Moreover,
$A_q(4,2)=q^4+q^3+2q^2+q+3$ for all $q$, $A_2(6,4)=A_2(6,4;3)=77$,
$q^6+2q^2+2q+1\le A_q(6,4)\le (q^3+1)^2$ for $q\ge 3$, see
\cite{MR3329980}, and $A_2(8,6)=A_2(8,6;4)=257$ \cite{paper257}.  The
$8$ isomorphism types of all latter optimal codes have been classified
in \cite{heinlein2018binary}. If $v=2m+1\ge 5$ is odd, then
$A_q(v,v-2)\in\{2q^{m+1}+1,2q^{m+1}+2\}$. Moreover,
$A_q(5,3) = 2q^3+2$ for all $q$ and $A_2(7,5) = 2\cdot 2^4+2=34$. The
$20$ isomorphism types of all latter optimal codes have been
classified in \cite{honold2016classification}.

Next we consider upper bounds for mixed dimension codes that are applicable for all parameters. Since the minimum subspace distance in a constant 
dimension code is even, decomposing the code into constant dimension codes gives $A_q(v,d)\le\sum\limits_{i=0}^v A_q(v,2\left\lceil d/2\right\rceil;i)$. 
Observing $A_q(v,d;\{0,1,\dots,\left\lceil d/2\right\rceil-1\})=A_q(v,d;\{v-\left\lceil d/2\right\rceil+1,\dots,v-1,v\})=1$, this was slightly tightened 
to $A_q(v,d)\le 2+\sum\limits_{i=\left\lceil d/2\right\rceil}^{v-\left\lceil d/2\right\rceil} A_q(v,2\left\lceil d/2\right\rceil;i)$ in 
\cite[Theorem 2.5]{honold2015constructions}. There is yet another tiny improvement, which seems to have been unnoticed so far:
\begin{lemma}
  \label{lemma_improved_cdc_upper_bound}
  If $\left\lceil d/2\right\rceil$ divides $v$, then 
  $
    A_q(v,d)\le \sum\limits_{i=\left\lceil d/2\right\rceil}^{v-\left\lceil d/2\right\rceil} A_q(v,2\left\lceil d/2\right\rceil;i)
  $.
\end{lemma}
\begin{proof}
  The constant dimension codes attaining $A_q(v,2\left\lceil d/2\right\rceil;\left\lceil d/2\right\rceil)$ are spreads, which cover each point 
  exactly once and hence have distance $<d$ to all subspaces of dimension
  $k\in\{0,1,\dots,\left\lceil d/2\right\rceil-1\}$ (and similarly for $i=v-\left\lceil d/2\right\rceil$).
\end{proof}
Let us remark that this lemma gives $A_2(6,3)\le 119$. By extending the known five isomorphism types of constant dimension codes attaining 
$A_2(6,4;3)=77$ one can reduce this bound by $1$ to $A_2(6,3)\le 118$, see \cite[Section 4.2]{honold2015constructions}. 
For the best known bounds on $A_2(v,d)$ with $v\le 8$ we refer the reader to \cite{heinlein2018binary}.

According to \cite{MR3063504} the, so far, only successful generalization of the classical bounds to projective space was 
given by Etzion and Vardy in \cite[Theorem 10]{MR2810308}. The approach generalizes the sphere-packing bound for constant dimension codes 
facing the fact that the spheres have different sizes. To that end let $B(U,e)$ denote the ball with center $U$ and radius $e$. 
Those balls around codewords are pairwise disjoint for subspace distance $d=2e+1$. Denoting the number of $k$-dimensional subspaces 
contained in $B(U,e)$ with $\dim(U)=i$ by 
$c(i,k,e)$, we have
\begin{equation*}
  c(i,k,e)=\sum_{j=\left\lceil\frac{i+k-e}{2}\right\rceil}^{\min\{k,i\}} \gaussmnum{i}{j}{q}\gaussmnum{v-i}{k-j}{q} q^{(i-j)(k-j)}.
\end{equation*} 
Thus, $A_q(v,2e+1)$ is at most as large as the target value of the ILP:
\begin{align}
\max \sum_{i=0}^v &a_i \label{ILP_EtzionVardy}&&\text{subject to}\\
a_i & \le A_q(v,2e+2;i)  && \forall 0\le i\le v \nonumber\\
\sum_{i=0}^v c(i,k,e)\cdot a_i & \le \gaussmnum{v}{k}{q} && \forall 0\le k\le v  \nonumber\\
&a_i \in \mathbb{N} &&\forall 0\le i\le v\nonumber
\end{align}
Here, the $a_i$ denote the number of codewords of dimension $i$. As for each ILP one can consider the LP relaxation, i.e., replacing the 
integer variables by non-negative real variables, in order to decrease computation times. For this ILP it turns out that the gap between 
the target value of the ILP and the corresponding LP is quite often smaller than $1$. Note that the described sphere-packing approach for even 
distances is obtained via $A_q(v,2e+2)\le A_q(v,2e+1)$, which nevertheless turns out to be the best known upper bound in some cases, see 
e.g.\ the bounds for $A_2(10,6)$ and $A_2(10,5)$ in \cite{TableSubspacecodes}.

As the problem of the determination of $A_q(v,d)$ can be naturally formulated as a maximum clique problem, and the Lov\'asz theta bound from
semidefinite programming can be applied. Since the problem size is linear in terms of the graph parameters, they are exponential in $v$. 
However, one can take the acting symmetry group into account in order to drastically decrease the problem size, see \cite{bachoc2012invariant} 
for general reduction techniques for invariant semidefinite programs. Obtaining parametric formulas for the reduced SDP formulations 
is a highly non-trivial task in general, and was achieved for $\vartheta'$ of the graph corresponding to $A_q(v,d)$ in \cite{MR3063504}.    
The authors report several numerical results for $q=2$ and odd distances. Where they are computed, this gives the best known upper bound 
in many cases. Using improved upper bounds for constant dimension codes, especially partial spreads, in \cite{heinlein2018new} the 
authors compute numerical results, also for $q>2$ and even distances. 

\section{Johnson type bounds for constant dimension codes}
\label{section_johnson_cdc}

One approach to obtain upper bounds for constant dimension codes is to try to generalize upper bounds for binary error-correcting 
constant weight codes in the Hamming metric, which corresponds to the case $q=1$. Several of the latter have been obtained by 
Johnson in~1962. The bound \cite[Theorem~3]{johnson1962new}, see also \cite{tonchev1998codes}, has been generalized by Xia and Fu to %:
\cite[Theorem~2]{xia2009johnson}. 
%% \begin{theorem}
%%   \label{thm_johnson_I}
%%   (\textbf{Johnson type bound I}) \cite[Theorem~2]{xia2009johnson}\\ If $\left(q^k-1\right)^2>\left(q^v-1\right)\left(q^{k-d/2}-1\right)$, then
%%   $
%%     A_q(v,d;k)\le \frac{\left(q^k-q^{k-d/2}\right)\left(q^v-1\right)}{\left(q^k-1\right)^2-\left(q^v-1\right)\left(q^{k-d/2}-1\right)}
%%   $.
%% \end{theorem} 
%% 
%% Indeed, Theorem~\ref{thm_johnson_I} is cited in several papers on constant dimension codes. However, it seems widely unnoticed, 
%%  that the formulation of the bound can be simplified considerably.
%% 
%% \begin{proposition}
%%   \label{prop_johnson_I}
%%   \cite[Proposition~1]{heinlein2017asymptotic}\\
%%   For $0\le k<v$ and $d\le 2\min\{k,v-k\}$, the bound in Theorem~\ref{thm_johnson_I} is applicable iff $d=2\min\{k,v-k\}$ and $k\ge 1$. Then, 
%%   it is equivalent to
%%   \[
%%     A_q(v,d;k)\le \frac{q^v-1}{q^{\min\{k,v-k\}}-1}.
%%   \]
%% \end{proposition}
However, the formulation of the bound can be simplified considerably, see \cite[Proposition~1]{heinlein2017asymptotic}, and only applies 
to partial spreads, i.e., $d=2k$. 
%%
%%So, Theorem~\ref{thm_johnson_I} is actually an upper bound for partial spreads arising from the simple observation that the number of codewords 
%%cannot be larger than the number of points of the ambient space divided by the number of points of a codeword. Except the trivial case $v=k$, 
%%where $A_q(v,d;v)=1$, there are strictly tighter bounds known, see \cite{kurz2017packing}.
While the generalization of \cite[Theorem~3]{johnson1962new} is rather weak, generalizing \cite[Inequality~(5)]{johnson1962new} yields a
considerably stronger upper bound: 
\begin{theorem}
  \label{thm_johnson_II}
  (\textbf{Johnson type bound II}) \cite[Theorem~3]{xia2009johnson}, \cite[Theorems~4 and~5]{MR2810308} 
  \begin{eqnarray}
  A_q(v,d;k) &\le& \left\lfloor 
  \frac{q^v-1}{q^k-1} A_q(v-1,d;k-1) \right\rfloor
  \label{ie_j_2}\\
  A_q(v,d;k) &\le& \left\lfloor 
  \frac{q^v-1}{q^{v-k}-1} A_q(v-1,d;k) \right\rfloor
  \label{ie_j_o}
  \end{eqnarray}
\end{theorem}
For the proof of Inequality~(\ref{ie_j_2}) one considers all codewords containing an arbitrary but fixed point. Since there are at most 
$A_q(v-1,d;k-1)$ such codewords and the number of points in the ambient space and in a codeword is $\gaussmnum{v}{1}{q}$ and $\gaussmnum{k}{1}{q}$, 
respectively, the upper bound follows. Inequality~(\ref{ie_j_o}) is obtained if one considers codewords contained in a given hyperplane instead.   
We remark that Inequality~(\ref{ie_j_2}) and Inequality~(\ref{ie_j_o}) are equivalent using duality, i.e., $A_q(v,d;k)=A_q(v,d;v-k)$.

Of course, the bounds in Theorem~\ref{thm_johnson_II} can be applied iteratively. %%In the classical Johnson space the optimal
For binary error-correcting constant weight codes in the Hamming metric the optimal  
choice of the corresponding inequalities is unclear, see e.g.\ \cite[Research Problem 17.1]{MR0465510}, while we have:
\begin{proposition}
  \cite[Proposition~2]{heinlein2017asymptotic}\\
  \label{prop_optimal_johnson}
  For $k \le v/2$ we have
  \[
    \left\lfloor \frac{q^v-1}{q^k-1} A_q(v-1,d;k-1) \right\rfloor
    \le
    \left\lfloor \frac{q^v-1}{q^{v-k}-1} A_q(v-1,d;k) \right\rfloor,
  \]
  where equality holds iff $v=2k$.
\end{proposition}
So, initially assuming $k\le v/2$, the optimal choice is to iteratively apply Inequality~(\ref{ie_j_2}), which results in: 
\begin{corollary}
\textbf{(Implication of the Johnson type bound II)}
\label{cor_johnson_opt}
\[
A_q(v,d;k)
\le
\left\lfloor \frac{q^{v}\!-\!1}{q^{k}\!-\!1} \left\lfloor \frac{q^{v\!-\!1}\!-\!1}{q^{k\!-\!1}\!-\!1} \left\lfloor \ldots 
\left\lfloor \frac{q^{v\!-\!k\!+\!d/2\!+\!1}\!-\!1}{q^{d/2\!+\!1}\!-\!1} A_q(v\!-\!k\!+\!d/2,d;d/2) \right\rfloor 
\ldots \right\rfloor \right\rfloor \right\rfloor.
\]
\end{corollary}
We prefer not to insert $A_q(v\!-\!k\!+\!d/2,d;d/2)\le \left\lfloor\frac{q^{v-k+d/2}-1}{q^{d/2}-1}\right\rfloor$, since 
currently much better bounds for partial spreads are available, which we will discuss next.

In the case $d=2k$, any two codewords of $\mathcal{C}$ intersect trivially, meaning that each point of $\PG(\F_q^v)$ is covered by at most a single codeword.
These codes are better known as \emph{partial $k$-spreads}.
If all the points are covered, we have $\#\mathcal{C} = \gauss{v}{1}{q}/\gauss{k}{1}{q}$ and $\mathcal{C}$ is called a \emph{$k$-spread}.
From the work of Segre in 1964 \cite[\S VI]{segre1964teoria} we know that $k$-spreads exist if and only 
if $k$ divides $v$. Upper bounds for the size of a partial $k$-spreads are due to Beutelspacher \cite{beutelspacher1975partial} and Drake \& Freeman 
\cite{nets_and_spreads} and date back to 1975 and 1979, respectively. Starting with \cite{kurz2016improved}, several recent improvements have been obtained. 
Currently the tightest upper bounds, besides $k$-spreads, are given by a list of $21$ sporadic $1$-parametric series and the following 
two theorems stated in \cite{kurz2017packing}:
\begin{theorem}
  \label{main_theorem_1_ps}
  For integers $r\ge 1$, $t\ge 2$, $u\ge 0$, and $0\le z\le \gauss{r}{1}{q}/2$ with $k=\gauss{r}{1}{q}+1-z+u>r$ we have
  $\smax_q(v,2k;k)\le lq^k+1+z(q-1)$, where $l=\frac{q^{v-k}-q^r}{q^k-1}$ and $v=kt+r$.   
\end{theorem}

\newcommand{\uu}{\lambda}
\begin{theorem}
  \label{main_theorem_2_ps}
  For integers $r\ge 1$, $t\ge 2$, $y\ge \max\{r,2\}$, %%$y\ge r+1$, 
  $z\ge 0$ with $\uu=q^{y}$, $y\le k$, %%$z+y-1\le\gauss{r}{1}{q}$, 
  $k=\gauss{r}{1}{q}+1-z>r$, $v=kt+r$, and  $l=\frac{q^{v-k}-q^r}{q^k-1}$, we have $\smax_q(v,2k;k)\le $
  $$
     lq^k+\left\lceil \uu -\frac{1}{2}-\frac{1}{2}
    \sqrt{1+4\uu\left(\uu-(z+y-1)(q-1)-1\right)} \right\rceil
  .$$   
\end{theorem}
 
The special case $z=0$ in Theorem~\ref{main_theorem_1_ps} covers the breakthrough $\smax_q(kt+r,2k;k)=1+\sum_{s=1}^{t-1}q^{sk+r}$ for 
$0<r<k$ and $k>\gauss{r}{1}{q}$ by N{\u{a}}stase and Sissokho \cite{nastase2016maximum} from 2016, which itself covers the result of Beutelspacher.  
The special case $y=k$ in Theorem~\ref{main_theorem_2_ps} covers the result by Drake \& Freeman. A contemporary survey of the best known upper bounds 
for partial spreads can be found in \cite{partial_spreads_and_vectorspace_partitions}.  

All currently known upper bounds for partial spreads can be deduced
from the non-existence of certain divisible codes, see
\cite{partial_spreads_and_vectorspace_partitions}.  The set $N$ of all
points of the ambient space not contained in any $k$-space of a
partial spread corresponds to a projective linear code over
$\mathbb{F}_q$ of length $n=\#N$ with all codeword Hamming weights
divisible by $q^{k-1}$. Recently this idea was generalized to constant
dimension codes with $d<2k$ in \cite{kiermaier2017improvement}. If
$\mathcal{C}$ is a set of subspaces with dimensions at least $r$ and
such that every point $P$ is contained in at most $j$ subspaces $X\in
\mathcal{C}$, then the multiset $N$ of points defined
  by $P\mapsto j-\#\{X\in\mathcal{C};P\in X\}$ corresponds to a possibly
non-projective linear code over $\mathbb{F}_q$ of length $\# N$ with
all codeword Hamming weights divisible by $q^{r-1}$. Let
$\left\{a / \gauss{k}{1}{q}\right\}_k$ denote the maximal
$b\in\mathbb{N}$ permitting a $q^{k-1}$-divisible code of length
$a-b\cdot \gauss{k}{1}{q}$ over $\mathbb{F}_q$. With this,
Inequality~(\ref{ie_j_2}) can be tightened (obviously we have
$\left\{a / \gauss{k}{1}{q}\right\}_k\le \left\lfloor a /
  \gauss{k}{1}{q}\right\rfloor$) to
$$
  A_q(v,d;k) \le \left\{ 
  \gauss{v}{1}{q}\cdot A_q(v-1,d;k-1)/ \gauss{k}{1}{q} \right\}_k.
$$
Using the abbreviation $v'=v-k+d/2$, the iterated application yields $A_q(v,d;k)$
\begin{equation}
  \label{ie_johnson_improved}
  \le \left\{ \frac{\gauss{v}{1}{q}}{\gauss{k}{1}{q}}\cdot\left\{\frac{\gauss{v-1}{1}{q}}{\gauss{k-1}{1}{q}}\cdot
  \left\{ \dots \cdot\left\{ \frac{\gauss{v'+1}{1}{q}}{\gauss{d/2+1}{1}{q}}  \cdot A_q(v',d;d/2) 
  \right\}_{d/2+1}\!\!\!\!\!\!\!\!\! \dots \right\}_{k-2} \right\}_{k-1}\right\}_k\text{,}
\end{equation}
which is the tightened version of Corollary~\ref{cor_johnson_opt}.

While the question whether a projective $q^r$-divisible linear code over $\mathbb{F}_q$ of length $n$ exists, 
is unsolved in general, this problem has been solved in the non-projective case via an efficient algorithm, see \cite[Theorem 4]{kiermaier2017improvement} and 
\cite[Algorithm 1]{kiermaier2017improvement}, i.e., $\left\{a / \gauss{k}{1}{q}\right\}_k$ can be computed efficiently. Results from the theory of 
$q^r$-divisible linear codes over $\mathbb{F}_q$ are exemplarily applied in Lemma~\ref{lemma_a_3_9_5}.

We remark that Inequality~(\ref{ie_johnson_improved}) combined with the best known upper bounds for partial spreads yields the best known upper 
bounds for constant dimension codes except for $\smax_2(6,4;3)=77<81$ \cite{MR3329980} and $\smax_2(8,6;4)=257\le 272<289$ 
\cite{paper257,heinlein2017new}. The mentioned improvements are based on extensive integer linear programming computations.
In contrast to that, the improvements in this article are based on self-contained theoretical arguments and do not need any huge computations.

\section{Johnson type bounds for mixed dimension codes}
\label{section_johnson_mdc}

Since Theorem~\ref{thm_johnson_II} (and its refinement based on
$q^r$-divisible codes) is that competitive for constant dimension
codes, it seems quite natural to investigate the underlying idea also
in the mixed dimension case. As the number of points in subspaces of
different dimensions is different, we have to take the precise
dimension distribution of those codewords that contain a specific
point into account. To that end let $F_q(v,d)$ be the set of
$(v+1)$-tuples $b=(b_0,\dots,b_v)\in\mathbb{N}^{v+1}$ such that 
there exists a mixed dimension code $\mathcal{C}$ in $\mathbb{F}_q^v$ with minimum distance
at least $d$ and dimension distribution $b$. Note that our
numbering starts from $0$. Transferring the idea of
the Johnson bound we end up with the following integer linear
programming (ILP) formulation

\begin{proposition}
\label{prop_ilp_1}
For all $v\ge 2$ and $1\le d\le v$ we have
\begin{align}
A_q(v,d)\le \max \sum_{i=0}^v &a_i \label{ILP_basic}&&\text{subject to}\\
\sum_{b \in F_q(v-1,d)} b_{i-1}x_b & = \gaussmnum{i}{1}{q} a_i && \forall 1\le i\le v \nonumber\\
\sum_{b \in F_q(v-1,d)} x_b & = \gaussmnum{v}{1}{q} && \nonumber\\
a_0 &\le 1\nonumber\\
A_q(v,d;i)a_0+a_i    &\le A_q(v,d;i) && \forall 1\le i\le d-1 \nonumber\\
&x_b \in \mathbb{N} &&\forall b \in  F_q(v-1,d)\nonumber\\
&a_i \in \mathbb{N} &&\forall 1\le i\le v\nonumber
\end{align}
\end{proposition}
\begin{proof}
  Let $\mathcal{C}$ be a subspace code in $\mathbb{F}_q^v$ with minimum subspace distance at least $d$ whose 
  cardinality attains $A_q(v,d)$. By $a_i\in\mathbb{N}$ we denote the number of $i$-dimensional codewords for $0\le i\le v$, i.e., 
  vector $a$ is the dimension distribution of $\mathcal{C}$. 
  For every point $P$ of $\mathbb{F}_q^v$ let $\mathcal{C}_P$ be $\{C\in\mathcal{C}\,:\, P\le C\}$ modulo $P$  
  and $b_P\in\mathbb{N}^v$ be the dimension distribution of $\mathcal{C}_P$. Thus $b_P\in F_q(v-1,d)$. By $x_b\in\mathbb{N}$ we 
  count the number of points $p$ such that $b=b_P$. Since each $i$-dimensional codeword of $\mathcal{C}$ contains $\gauss{i}{1}{q}$ 
  points we have $\sum_{b \in F_q(v-1,d)} b_{i-1}x_b  = \gaussmnum{i}{1}{q} a_i$ for all $1\le i\le v$. Since every $b_P$ is 
  counted exactly once, we have $\sum_{b \in F_q(v-1,d)} x_b  = \gaussmnum{v}{1}{q}$. Of course $\mathcal{C}$ can contain at most 
  one $0$-dimensional codeword. Since $a_0$ is not coupled with the $x_b$-variables, we use the fact that $\mathcal{C}$ can not contain 
  both a $0$- and an $i$-dimensional codeword for $1\le i\le d-1$. This can be modeled as $A_q(v,d;i)a_0+a_i    \le A_q(v,d;i)$, i.e., 
  if $a_0=1$ then the inequality reads $a_i\le 0$ and if $a_0=0$ then the inequality is equivalent to $a_i \le A_q(v,d;i)$, which 
  is also valid. 
\end{proof}

We remark that the hard-to-compute values $A_q(v,d;i)$, occurring as coefficients of inequalities in the above ILP, may be replaced by any 
upper bound on $A_q(v,d;i)$ and the set $F_q(v-1,d)$ may be also replaced by any superset, which of course may weaken the resulting upper bound. 
Of course, further inequalities like e.g.\ $\sum_{i\in K} a_i\le A_q(v,d;K)$ for some $K\subseteq \{0,\dots,v\}$ may be added. 

Having Proposition~\ref{prop_optimal_johnson} at hand, it is obvious that ILP formulation 
(\ref{ILP_basic}) can be further improved by also taking Inequality~(\ref{ie_j_o}) into account. However, instead of considering subcodes in 
hyperplanes we use duality in order to assume 
$\sum_{i=0}^{\left\lfloor v/2\right\rfloor} a_i \ge \sum_{\left\lceil v/2\right\rceil}^{v} a_i$, which also allows us to eliminate variables. 
As a shortcut, we use $m=\left\lfloor v/2\right\rfloor$ in the following and denote by $\overline{F}_q(v-1,d)$ the set of $m$-tuples 
$b=(b_0,\dots,b_{m-1})\in\mathbb{N}^{m}$ such that there exists a subspace code $\mathcal{C}$ in $\mathbb{F}_q^{v-1}$ with minimum subspace 
distance, at least, $d$, codewords with dimensions in $\{0,\dots,m-1\}$, and dimension distribution $b$. With this we can directly 
reformulate Proposition~\ref{prop_ilp_1} to: 
\begin{proposition}
\label{prop_ilp_2}
For all $v\ge 2$ and $1\le d\le v$ we have
\begin{align}
A_q(v,d)\le \max t(a) & \label{ILP_half}&&\text{subject to}\\
\sum_{b \in \overline{F}_q(v-1,d)} b_{i-1}x_b & = \gaussmnum{i}{1}{q} a_i && \forall 1\le i\le m \nonumber\\
\sum_{b \in \overline{F}_q(v-1,d)} x_b & = \gaussmnum{v}{1}{q} && \nonumber\\
a_0 &\le 1\nonumber\\
A_q(v,d;i)a_0+a_i    &\le A_q(v,d;i) && \forall 1\le i\le \min\{d-1,m\} \nonumber\\
&x_b \in \mathbb{N} &&\forall b \in  \overline{F}_q(v-1,d)\nonumber\\
&a_i \in \mathbb{N} &&\forall 0\le i\le m\nonumber,
\end{align}
where $m=\left\lfloor v/2\right\rfloor$ and $t(a)=2\sum\limits_{i=0}^m a_i$ if $v-1$ is even and by $t(a)=a_m+2\sum\limits_{i=0}^{m-1} a_i$ otherwise. 
\end{proposition}

Now let us compute the contribution of a point $P$ with dimension distribution $b_P$ of the corresponding code $\mathcal{C}_P$. 
For $b=b_P\in\mathbb{N}^m$ and $m=\lfloor v/2\rfloor$ we set
\begin{equation}
  \Gamma_v(b)=\left\{\begin{array}{rcl}2\sum\limits_{i=0}^{m-1} \frac{b_i}{\gauss{i+1}{1}{q}}&:&v\equiv 1\pmod 2,\\
  \frac{b_{m-1}}{\gauss{m}{1}{q}}+2\sum\limits_{i=0}^{m-2} \frac{b_i}{\gauss{i+1}{1}{q}}&:& v\equiv 0\pmod 2.\end{array}\right.
\end{equation}
and call $\Gamma_v(b)$ \emph{score} of $b$. With this we have
\begin{equation}
  \label{eq_score_bound}
  t(a)=\sum_b x_b\Gamma_v(b)\le \gauss{v}{1}{q}\cdot\max_b \Gamma_v(b) 
\end{equation}
In other words, we express the target function in terms of the $x_b$ and ignore all constraints on the $a_i$, giving an easy upper bound.

Even for small parameters $v$ and $d$ the sets $\overline{F}_q(v-1,d)$ can become quite large, so that we introduce another ILP variant. 
To that end let us write $c'\le c$ for two vectors $c',c\in\mathbb{R}^n$ iff $c'_i\le c_i$ for all $1\le i\le n$, where the integer $n$ will be 
always clear from the context. Of course we have $\Gamma_v(c')\le \Gamma_v(c)$. Let $\overline{F}_q(v-1,d)\subset\mathcal{F}$ for some set 
$\mathcal{F}$. We call an element $f$ of $\mathcal{F}$ \emph{maximal} if there does not exist an element $f'\in\mathcal{F}$ with 
$f'\ge f$ and $f'\neq f$. If $\overline{\mathcal{F}}$ contains all maximal elements from $\mathcal{F}$ then we can restrict to 
$b\in \overline{\mathcal{F}}$ in Proposition~\ref{prop_ilp_2} if we replace $\sum_{b \in \mathcal{F}} b_{i-1}x_b  = \gaussmnum{i}{1}{q} a_i$ 
by $\sum_{b \in \overline{\mathcal{F}}} b_{i-1}x_b  \ge \gaussmnum{i}{1}{q} a_i$.    

Combining both ideas gives:
\begin{proposition}
\label{prop_ilp_3}
Let $v\ge 2$ and $1\le d\le v$ be integers, $\omega\in\mathbb{R}_{ge 0}$, and $m=\left\lfloor v/2\right\rfloor$. 
If $\mathcal{F}\subseteq \mathbb{N}^m$ such that for all $f'\in \overline{F}_q(v-1,d)$ either there exists an 
element $f\in \mathcal{F}$ with $f\ge f'$ or $\Gamma_v(f')\le \omega$, then we have  
\begin{align}
A_q(v,d)\le \max &\,\,\omega z+t(v) \label{ILP_core}&&\text{subject to}\\
\sum_{b \in \mathcal{F}} b_{i-1}x_b & \ge \gaussmnum{i}{1}{q} a_i && \forall 1\le i\le m \nonumber\\
z+\sum_{b \in \mathcal{F}}x_b & = \gaussmnum{v}{1}{q}&& \nonumber\\
A_q(v,d;i)a_0+a_i    &\le A_q(v,d;i) && \forall 1\le i\le \min\{d-1,m\} \nonumber\\
a_i    &\le A_q(v,d;i) && \forall \min\{d-1,m\}+1\le i\le m \nonumber\\
&x_b \in \mathbb{N} &&\forall b \in  \mathcal{F}\nonumber\\
&z\in\mathbb{N},\nonumber
\end{align}
where $t(a)$ is defined as in Proposition~\ref{prop_ilp_2}.
\end{proposition}
\begin{proof}
  We extend the ILP model from Proposition~\ref{prop_ilp_2} by counting $b_P$ either in $x_b$ where $b_P\le b$ and 
  $b\in\mathcal{F}$ or in the new auxiliary variable $z$ (then $\Gamma_v(b_P)\le \omega$). The interpretation 
  of the $a_i$ changes slightly if $z>0$ since some contributions of the $b_P$ are hidden in $z$.
\end{proof}
Note that we can add the restrictions $a_i\in\mathbb{N}$ if $z=0$, i.e., the $a_i$ keep their meaning as the dimension distribution of 
the code $\mathcal{C}$.

In the following we will mostly use the ILP formulation~(\ref{ILP_core}) in order to compute improved upper bounds for $A_q(v,d)$. It remains to 
provide an algorithm to compute a feasible set $\mathcal{F}$. For all $b\in\overline{F}_q(v-1,d)$ we obviously have $0\le b_i\le A_q(v-1,d;i)$, so 
that there is only a finite number of possibilities. In order to check whether $b\in\overline{F}_q(v-1,d)$ we slightly modify 
the ILP formulation~(\ref{ILP_core}) by setting $z=0$, replacing $v$ by $v-1$, adding the constraints $a_i\ge b_i$ for all $i\neq j$, and replacing 
the target function by $a_j$, where $j$ can be chosen freely. If there is no solution, then $b\notin \overline{F}_q(v-1,d)$. Otherwise the 
solution vector $a$ can be added to $\mathcal{F}$ and all $b'\le a$ do not need to be considered any more. Moreover, all vectors 
where $a_j$ is replaced by a larger number in $a$ cannot be contained in $\overline{F}_q(v-1,d)$. This gives a recursive algorithm, which works 
in principle. For \textit{larger} parameters, it will become computationally infeasible. However, by a mixture between theoretical reasoning and 
(I)LP computations we will be able to determine suitable sets $\mathcal{F}$ for many parameters. In the determination of $\mathcal{F}$ we will 
speak of \emph{maximal patterns}.     

We give a concrete numerical example:
\begin{lemma}
  \label{lemma_a_2_10_5}
  $A_2(10,5)\le 48104$  
\end{lemma}
%% can most likely be improved by 1
\begin{proof}
  Let $\mathcal{C}$ be a subspace code of $\mathbb{F}_2^{10}$ with minimum subspace distance $d=5$, and let $a_i$ denote the number of 
  $i$-dimensional codewords. W.l.o.g.\ we assume $a_0+a_1+a_2+a_3+a_4\ge a_6+a_7+a_8+a_9+a_{10}$, so that $\#\mathcal{C}\le a_5+2 \sum_{i=0}^4 a_i$.
  If $a_0+a_1\ge 1$, then $a_0+a_1=1$ and $\#\mathcal{C}\le 2+2a_4+a_5\le 2+2\cdot 4977+38148=48104$ due to $A_2(10,5;4)\le 4977$ and 
  $A_2(10,5;5)\le 38148$. Thus, we assume $a_0=a_1=0$ in the following. 
  
  Next we consider the possible maximal patterns of codewords through
  a point, i.e., in $\mathbb{F}_2^{9}$ we consider sets of codewords
  with dimension in $\{1, 2, 3,4\}$ and minimum distance at
  least $5$. Since $A_2(9,6;3)=73$ and $A_2(9,6;4)\le 1156$ the
  maximal patterns are below $1^1 4^{1156}$, $3^{73} 4^{1156}$, or
  patterns of the form $2^1 3^{x_3} 4^{x_4}$. So, let us determine
  bounds for $x_3$ and $x_4$. In $\mathbb{F}_2^8$ is suffices to
  consider the patterns $1^1$ and $2^1 3^{34}$ since
  $A_2(8,5;1)=A_2(8,5;2)=1$ and $A_2(8,5;3)=34$. Only pattern
  $2^1 3^{34}$ contributes to $x_3$ or $x_4$. Since a line is present,
  we need at least $3$ times pattern $1^1$, so that at most
  $\gauss{9}{1}{2}-3=508$ points in $\mathbb{F}_2^9$ can have pattern
  $2^1 3^{34}$. Thus we have
  $x_3\le \left\lfloor 508/\gauss{3}{1}{2}\right\rfloor=72$ and
  $x_4\le\left\lfloor 508\cdot 34/\gauss{4}{1}{2}\right\rfloor=1151$.
  
  Finally, let
  $$
    \mathcal{F}=\left\{(0,1,0,0,1156),(0,0,0,73,1156),(0,0,1,72,1151)\right\}
  $$
  and apply the ILP of Proposition~\ref{prop_ilp_3} with $z=0$. This gives $\#\mathcal{C}\le 48104$. An optimal solution 
  is given by $a_3=3$, $a_4=4977$ and $b_5=38144$, where the second pattern is chosen $999$ and the third pattern is chosen $24$ 
  times.
%% Maximize target:
%% 2 a0 +2 a1 +2 a2 +2 a3+2 a4 +a5
%% Subject To
%% a0 = 0
%% a1 = 0
%% a4 <= 4977
%% b1+b2+b3=1023
%% b1 -3 a2 >=0
%% b3 -7 a3 >= 0
%% 73 b2+72 b3-15 a4 >= 0
%% 1156 b1+ 1156 b2 + 1151 b3 -31 a5 >= 0
%% General
%% a0
%% a1
%% a2
%% a3
%% a4
%% a5
%% b1
%% b2
%% b3
%% End
\end{proof} 
Note that the scores $\Gamma_{10}$ of the three patterns in $\mathcal{F}$ are less than $37.95699$, $47.023656$, 
and $47.014747$, respectively, so that Inequality~(\ref{eq_score_bound}) would give $\#\mathcal{C}\le \lfloor 48105.3 \rfloor
=48105$, i.e., the solution of the ILP slightly pays off.  

Classification and existence results for $q^r$-divisible codes can also be used to decrease upper bounds in the context 
of subspace codes with mixed dimensions. A concrete
numerical example is the following:
\begin{lemma}
  \label{lemma_a_3_9_5}
  $A_3(9,5)\le 123048$  
\end{lemma}
%% can be further improved
\begin{proof}
  Let $\mathcal{C}$ be a subspace code of $\mathbb{F}_3^9$ with minimum subspace distance $d=5$, and let $a_i$ denote the number of 
  $i$-dimensional codewords. W.l.o.g.\ we assume $a_0+a_1+a_2+a_3+a_4\ge a_5+a_6+a_7+a_8+a_9$, so that $\#\mathcal{C}\le 2 \sum_{i=0}^4 a_i$.
  If $a_0+a_1\ge 1$, then $a_0+a_1=1$ and $\#\mathcal{C}\le 2+2a_4\le 2+2A_3(9,6;4)\le 122022$. In the following we assume $a_0=a_1=0$. 
  
  Next we consider the possible maximal patterns of codewords through a point, i.e., in $\mathbb{F}_q^{8}$ we consider sets of 
  codewords with a dimension in $\{1, 2, 3\}$ and minimum subspace distance at least $5$: $1^1$ and $2^{x_2}3^{x_3}$, 
  where $x_2\le 1$ and $x_3\le A_3(8,6;3)$. For the latter the tightest known bounds are $244\le A_3(8,6;3)\le 248$. If 
  $A_3(8,6;3)=248$ then the corresponding $56$ holes have to form a $3^2$-divisible projective set for which the 
  unique possibility is the Hill cap, see e.g.\ \cite{partial_spreads_and_vectorspace_partitions}, which does not contain a 
  line. So, no vector space partition of type $1^{52} 2^1 3^{248}$ exists in $\mathbb{F}_3^8$, which implies $x_2+x_3\le 248$.  

  Solving ILP~(\ref{ILP_core}) with the patterns $1^1$, $2^1 3^{247}$, $3^{248}$ and $a_0=a_1=0$, $a_2\le 1$, $a_3\le 757$, $a_4\le 61010$ 
  gives the unique solution $a_2=0$, $a_3=757$, $a_4=60768$ with target value $123050$. 
    
  Assume for a moment that $a_3=757$. In that case the $757$ planes form a spread, i.e., each point is covered exactly once. 
  So each point can be contained in at most $247$ solids. Thus, $a_4\le \left\{ 247\cdot \gauss{9}{1}{3}/\gauss{4}{1}{3}\right\}_4=60766$. 
  (We have $a_4\le \lfloor 247\cdot \gauss{9}{1}{3}/\gauss{4}{1}{3}\rfloor=60768$. If $a_4=60768$ then there would be a $3^3$-divisible 
  linear code of length $7$, If $a_4=60767$ then there would be a $3^3$-divisible linear 
  code of length $47$, which both do not exist, see \cite{kiermaier2017improvement}.) Thus, $\#\mathcal{C}\le 123048$.
  
  If we add $a_3\le 756$ to our ILP formulation we also get a target value of $123048$.
\end{proof} 
%% Maximize target:
%% 2 a0 +2 a1 +2 a2 +2 a3+2 a4
%% Subject To
%% a0 = 0
%% a1 = 0
%% a2 <=1
%% a3 <= 757
%% a3 <= 756
%% a4 <= 61010
%% b1+b2+b3=9841
%% b1 -4 a2 >=0
%% b2 -13 a3 >= 0
%% 247 b2 +248 b3-40 a4 >= 0
%% General
%% a0
%% a1
%% a2
%% a3
%% a4
%% b1
%% b2
%% b3
%% End

We remark that this is a numerical improvement of the more general Lemma~\ref{lemma_n_minus_4_odd}. In the next section 
we apply the underlying general idea of the Johnson bound, as outlined above, to $A_q(v,v-4)$ and $A_q(8,3)$.

\section{Analytical results}
\label{sec_analytical_results}

\begin{lemma}
  \label{lemma_n_minus_4_odd}
  For odd $v\ge 7$ we have 
  \begin{eqnarray*}
    A_q(v,v-4)&\le&
    \max\Big\{
      2A_q(v,v-3;m-1)+2A_q(v,v-3;m),\\
      && 2+2\left\lfloor\left(\gauss{2m+1}{1}{q}-\gauss{m-2}{1}{q}\right)/\gauss{m-1}{1}{q}\right\rfloor\\
      &&+2\left\lfloor\left(\gauss{2m+1}{1}{q}-\gauss{m-2}{1}{q}\right)\cdot A_q(2m,2m-2;m-1)/\gauss{m}{1}{q}\right\rfloor  
    \Big\},
  \end{eqnarray*}
  where $m=(v-1)/2$.
\end{lemma}
\begin{proof}
  Let $\mathcal{C}$ be a subspace code of $\mathbb{F}_q^v$ with minimum subspace distance $d=v-4$, $m=\frac{v-1}{2}\ge 3$, 
  and $a_i$ denote the number of $i$-dimensional codewords. W.l.o.g.\ we assume $\sum_{i=0}^m a_i\ge \sum_{i=m+1}^v a_i$, so 
  that $\#\mathcal{C}\le 2 \sum_{i=0}^m a_i$. Since $d=2m-3$ we have $\sum_{i=0}^{m-2} a_i\le 1$. If there exists an index 
  $0\le i\le m-3$ with $a_i>0$, then $\#\mathcal{C}\le 2+2A_q(v,v-3;m)$. If $\sum_{i=0}^{m-2} a_i=0$, then $\#\mathcal{C}\le 
  2A_q(v,v-3;m-1)+2A_q(v,v-3;m)$, which is at least as large as $2+2A_q(v,v-3;m)$. It remains to consider the case $a_{m-2}=1$. 
  Here we consider the possible maximal patterns of codewords through a point, i.e., in $\mathbb{F}_q^{2m}$ we consider sets of 
  codewords with a dimension in $\{m-3, m-2, m-1\}$ and minimum subspace distance at least $2m-3$: $(m-3)^1$ and 
  $(m-2)^1 (m-1)^{A_q(2m,2m-2;m-1)}$. The first pattern is attained $\gauss{m-2}{1}{q}$ times so that
  $$ 
    a_{m-1}\le\left\lfloor\left(\gauss{2m+1}{1}{q}-\gauss{m-2}{1}{q}\right)/\gauss{m-1}{1}{q}\right\rfloor
  $$
  and
  $$
    a_m\le \left\lfloor\left(\gauss{2m+1}{1}{q}-\gauss{m-2}{1}{q}\right)\cdot A_q(2m,2m-2;m-1)/\gauss{m}{1}{q}\right\rfloor.
  $$ 
\end{proof}

We remark that $\left\lfloor\left(\gauss{2m+1}{1}{q}-\gauss{m-2}{1}{q}\right)/\gauss{m-1}{1}{q}\right\rfloor$ can be simplified to 
$q^5+q^3+q$ for $m=3$, $q^6+q^3$ for $m=4$, $q^7+q^3$ for $m=5$, and $q^{m+2}+q^3-1$ for $m\ge 6$. The upper 
bound can be improved if there is no vector space partition of type $1^\star (m-2)^1 (m-1)^{A_q(2m,2m-2;m-1)}$ of $\mathbb{F}_q^{2m}$.  
This happens e.g.\ for $m=3$ and arbitrary $q$. Since $A_q(6,4;2)=q^4+q^2+1$ and $A_q(7,4;2)=q^5+q^3+1$ the 
upper bound of upper bound of Lemma~\ref{lemma_n_minus_4_odd} evaluates to $A_q(7,3)\le 2(q^8+q^6+2q^5+q^4+2q^3+q^2+2)$ using 
the Anticode bound $A_q(7,4;3)\le\gauss{7}{1}{q}\cdot(q^2-q+1)$, which is the tightest known bound for these parameters. This 
can be further improved to:
\begin{lemma}
  \label{lemma_a_q_7_3}
  $A_q(7,3)\le 2(q^8+q^6+2q^5+2q^3+q^2-q+2)$  
\end{lemma}
\begin{proof}
  Let $\mathcal{C}$ be a subspace code of $\mathbb{F}_q^7$ with minimum subspace distance $d=3$, and let $a_i$ denote the number of 
  $i$-dimensional codewords. W.l.o.g.\ we assume $a_0+a_1+a_2+a_3\ge a_4+a_5+a_6+a_7$, so that $\#\mathcal{C}\le 2 \sum_{i=0}^3 a_i$.
  If $a_0\ge 1$, then $a_0=1$ and $\#\mathcal{C}\le 2+2a_3\le 2(q^8+q^6+q^5+q^4+q^3+q^2+2)\le 2(q^8+q^6+2q^5+2q^3+q^2-q+2)$.
  Next we consider the possible maximal patterns of codewords through a point, i.e., in $\mathbb{F}_q^{6}$ we consider sets of 
  codewords with a dimension in $\{0, 1, 2\}$ and minimum subspace distance at least $3$: $0^1$ and $1^{x_1}2^{x_2}$, where
  $x_1\le 1$ and $x_2\le A_q(6,4;2)=q^4+q^2+1$. Since $x_2=A_q(6,4;2)$ can only be attained in case of a line spread, we 
  have $x_1+x_2\le q^4+q^2+1$, which gives the possible maximal patterns $1^1 2^{q^4+q^2}$ and $2^{q^4+q^2+1}$. We start with the case 
  $a_1=0$ and denote the multiplicities of the patterns $1^1 2^{q^4+q^2}$ and $2^{q^4+q^2+1}$ by $m_1$ and $m_2$, respectively. With 
  this we have $a_2\le \left\lfloor m_1/\gauss{2}{1}{q}\right\rfloor$ and 
  $a_3\le \left\lfloor \left(\gauss{7}{1}{q}\cdot (q^4+q^2+1)-m_1\right)/\gauss{3}{1}{q} \right\rfloor$, so that  
  \begin{eqnarray*}
    \#\mathcal{C}&\le& 2 \left(m_1/\gauss{2}{1}{q}+\left(\gauss{7}{1}{q}\cdot (q^4+q^2+1)-m_1\right)/\gauss{3}{1}{q}\right) \\
    &=& 2m_1\cdot\frac{q^2}{(q+1)(q^2+q+1)}+2\gauss{7}{1}{q}\cdot (q^2-q+1)=:f(m_1),
  \end{eqnarray*}      
  which is increasing in $m_1$. Next we invoke $\#\mathcal{C}\le 2a_2+2a_3$ and $a_2\le A_q(7,4;2)=q^5+q^3+1$. 
  If $m_1\ge \gauss{2}{1}{q}\cdot (q^5+q^3+1)$, then $a_2\le q^5+q^3+1$ and 
  \begin{eqnarray*}
    a_3&\le& \left\lfloor \left(\gauss{7}{1}{q}\cdot (q^4+q^2+1)-\gauss{2}{1}{q}\cdot (q^5+q^3+1)\right)/\gauss{3}{1}{q} \right\rfloor\\
    &=& \gauss{7}{1}{q}\cdot(q^2-q+1)-\left\lceil\frac{(q+1)\cdot (q^5+q^3+1)}{q^2+q+1} \right\rceil\\
    &=& \gauss{7}{1}{q}\cdot(q^2-q+1)-q^4-q +\left\lfloor\frac{q^2-1}{q^2+q+1} \right\rfloor\\
    &=& q^8+q^6+q^5+q^3+q^2-q+1,
  \end{eqnarray*}   
  so that $\#\mathcal{C}\le 2(q^8+q^6+2q^5+2q^3+q^2-q+2)$. If $\gauss{2}{1}{q}\cdot (q^5+q^3)\le m_1<\gauss{2}{1}{q}\cdot (q^5+q^3+1)$, 
  then $a_2\le A_2(7,4;2)-1=q^5+q^3$ so that
  $$
    \#\mathcal{C}\le f((q+1)\cdot (q^5+q^3))=2\left(q^8+q^6+2q^5+2q^3+q^2-q+2-\frac{1}{q^2+q+1}\right). 
  $$  
  If $a_1=1$, then 
  \begin{eqnarray*}
    \#\mathcal{C}&\le& 2 \left(1+m_1/\gauss{2}{1}{q}+\left(\gauss{7}{1}{q}\cdot (q^4+q^2+1)-m_1-(q^4+q^2+1)\right)/\gauss{3}{1}{q}\right) \\
    &=& 2\left(m_1\cdot\frac{q^2}{(q+1)(q^2+q+1)}+\gauss{7}{1}{q}\cdot (q^2-q+1)-(q^2-q)\right).
  \end{eqnarray*}
  Since we can assume $m_1\le (q+1)(q^5+q^3+1)$ we have $\#\mathcal{C}\le 2(q^8+q^6+2q^5+2q^3+3)$.
\end{proof}

In the binary case Lemma~\ref{lemma_a_q_7_3} gives the upper bound 
$A_2(7,3)\le 808$ while the semidefinite programming method from \cite{MR3063504} gives $A_2(7,3)\le 776$.  
Also for $3\le q\le 7$ the semidefinite programming method gives tighter upper bounds, see \cite{heinlein2018new}.

\begin{lemma}
  \label{lemma_n_minus_4_even}
  Let $m\ge 4$. If $A_q(2m,2m-4)>2+A_q(2m,2m-4;m)$, then we have
  $$
    A_q(2m,2m-4)\le \left\lfloor \frac{\gauss{2m}{1}{q}}{\gauss{m}{1}{q}}\cdot \left\lfloor \frac{\left(\gauss{2m-1}{1}{q}-\gauss{m-3}{1}{q}\right)
    \cdot A_q(2m-2,2m-4;m-2)}{\gauss{m-1}{1}{q}} \right\rfloor
    +\frac{2\gauss{2m}{1}{q}}{\gauss{m-2}{1}{q}} \right\rfloor
  $$
  if $m=4$ or ($m=5$ and $q=2$) and
  $$
    A_q(2m,2m-4)\le \left\lfloor\frac{\gauss{2m}{1}{q}}{\gauss{m}{1}{q}}\cdot \left\lfloor \frac{\gauss{2m-1}{1}{q}\cdot A_q(2m-2,2m-4;m-2)}{\gauss{m-1}{1}{q}} \right\rfloor\right\rfloor 
  $$
  otherwise.
\end{lemma}
\begin{proof}
  Let $\mathcal{C}$ be a subspace code of $\mathbb{F}_2^{2m}$ with minimum subspace distance $d=2m-4$, and let $a_i$ denote the number of 
  $i$-dimensional codewords, so that $\#\mathcal{C}=\sum_{i=0}^{2m} a_i$. Due to duality we assume $\sum_{i=0}^{m-1} a_i\ge 
  \sum_{i=m+1}^{2m}a_i$, so that $\#\mathcal{C}\le a_m+2\sum_{i=0}^{m-1} a_i$. If $a_i\ge 1$ for an index $0\le i\le m-4$, then 
  $\#\mathcal{C}\le 2+A_q(2m,2m-4)$, so that we assume $a_i=0$ for all $0\le i\le m-4$ in the following. 
  Next we consider the possible maximal patterns of codewords through a point, i.e., in $\mathbb{F}_q^{2m-1}$ we consider sets of 
  codewords with a dimension in $\{m-4,m-3,m-2,m-1\}$ and minimum subspace distance at least $2m-4$: $(m-4)^1$, $(m-3)^1(m-1)^x$, 
  and $(m-2)^a (m-1)^b$, where we have to determine the possible values for $x$, $a$, and $b$. To that end we consider 
  the possible maximal patterns of codewords in $\mathbb{F}_q^{2m-2}$ with a dimension in $\{m-4,m-3,m-2\}$ and minimum subspace 
  distance at least $2m-4$: $(m-4)^1$, $(m-3)^1$, and $(m-2)^{\tau}$, where $\tau:=A_q(2m-2,2m-4;m-2)$. Thus, we can choose
  $$
    x=\left\lfloor \frac{\left(\gauss{2m-1}{1}{q}-\gauss{m-3}{1}{q}\right)\cdot\tau}{\gauss{m-1}{1}{q}} \right\rfloor
  $$
  and have
  $$
    b\le \left\lfloor \frac{\left(\gauss{2m-1}{1}{q}-a\gauss{m-2}{1}{q}\right)\cdot\tau}{\gauss{m-1}{1}{q}} \right\rfloor,
  $$
  where $a\in\mathbb{N}$. Since $\gauss{m-2}{1}{q}\ge 2$ and $\tau=A_q(2m-2,2m-4;m-2)\ge q^m\ge \gauss{m}{1}{q}$ we have 
  $$
    \frac{2a}{\gauss{m-1}{1}{q}}-\frac{\tau\gauss{m-2}{1}{q}\cdot a}{\gauss{m-1}{1}{q}\cdot\gauss{m}{1}{q}}\le 0 
  $$
  so that the score for pattern $(m-2)^a (m-1)^b$ is decreasing in $a$. For $a=0$ we obtain a score of 
  $$
    s_3:=\frac{1}{\gauss{m}{1}{q}}\cdot \left\lfloor \frac{\gauss{2m-1}{1}{q}\cdot\tau}{\gauss{m-1}{1}{q}} \right\rfloor.
  $$
  For pattern $(m-3)^1(m-1)^x$ the score is given by
  $$
    s_2:=\frac{1}{\gauss{m}{1}{q}}\cdot \left\lfloor \frac{\left(\gauss{2m-1}{1}{q}-\gauss{m-3}{1}{q}\right)\cdot\tau}{\gauss{m-1}{1}{q}} \right\rfloor
    +\frac{2}{\gauss{m-2}{1}{q}}
  $$
  and for $(m-4)^1$ we have a score of $s_1:=\frac{2}{\gauss{m-3}{1}{q}}$. In order to compare the three scores we use 
  $q^m+1\le \tau=A_q(2m-2,2m-4;m-2)\le \gauss{2m-2}{1}{q}/\gauss{m-2}{1}{q}$ and $q^{k-1}+1\le \gauss{k}{1}{q}\le 2 q^{k-1}-1$ for $k\ge 2$. 
  Obviously, we have $s_3\ge s_1$.
  If $m=4$, then using $\gauss{1}{1}{q}=1$ gives
  \begin{eqnarray*}
    s_2-s_3 &\ge &  -\frac{1}{\gauss{m}{1}{q}}\cdot  \frac{\gauss{m-3}{1}{q}\cdot\tau+1}{\gauss{m-1}{1}{q}}  +\frac{2}{\gauss{m-2}{1}{q}}\\
            &\ge&  \frac{1}{\gauss{m-2}{1}{q}}\cdot\left(2-\frac{2q^{2m-3}}{q^{m-1}q^{m-2}}\right)\ge 0.
  \end{eqnarray*}
  In the other direction we have
  \begin{eqnarray*}
    s_2-s_3 &\le&  -\frac{1}{\gauss{m}{1}{q}}\cdot \frac{\gauss{m-3}{1}{q}\cdot\tau-1}{\gauss{m-1}{1}{q}}   +\frac{2}{\gauss{m-2}{1}{q}}\\
            &\le&  \frac{2}{q^{m-3}}- \frac{1}{2q^{m-1}}\cdot\left(\frac{q^{m-4}\cdot q^{m}}{2q^{m-2}} -1\right)\\
            &\le & \frac{2}{q^{m-3}}\cdot\left(1-\frac{1}{4q^2}\cdot \left(\frac{q^{m-2}}{2}-1\right)\right),
  \end{eqnarray*}
  which is negative if $m\ge 7$ or $m=6$ and $q\ge 3$. For $m=6$ and $q=2$ we plug in the known numerical 
  values for the first inequality and obtain $s_2-s_3<0$. It remains to consider $m=5$, where $A_q(8,6,3)\ge q^5+1$ and
  \begin{eqnarray*}
    s_2-s_3 &\le &  -\frac{1}{\gauss{5}{1}{q}}\cdot \left\lfloor \frac{(q+1)\tau-1}{q^3+q^2+q+1} \right\rfloor  +\frac{2}{q^2+q+1}\\
            &\le& -\frac{q^3-q}{q^4+q^3+q^2+q+1}+\frac{2}{q^2+q+1},
  \end{eqnarray*}   
  which is negative for $q\ge 3$. For $m=5$ and $q=2$ we can easily check have $s_2-s_3>0$.
\end{proof}

The next lemma shows that a specific configuration consisting of a point, some lines and some planes does not exist in $\mathbb{F}_q^7$. 
This result will then be used to proof an upper bound on $A_q(8,3)$.   

\begin{lemma}
  \label{lemma_excluded_1}
  There exists no subspace code $\mathcal{C}$ in $\mathbb{F}_q^7$ with minimum subspace distance $d=3$ and dimension distribution $1^1 2^{q^4 + q^2 + 2} 3^{q^8+q^6+q^5+q^3}$.
\end{lemma}

\begin{proof}
  Assume that $\mathcal{C}$ is a code in $V = \F_q^7$ of minimum subspace distance $3$ containing a single point $P$ and $q^8+q^6+q^5+q^3$ planes.
  We denote the set of lines in $\mathcal{C}$ by $\mathcal{C}_2$ and the set of planes in $\mathcal{C}$ by $\mathcal{C}_3$.
  As the subspace distance is at least $3$, $P$ is not contained in any element of $\mathcal{C}_2$ and $\mathcal{C}_3$, no line in $\mathcal{C}_2$ is contained in a plane in $\mathcal{C}_3$, the lines in $\mathcal{C}_2$ are pairwise disjoint and the pairwise intersection of the planes in $\mathcal{C}_3$ is at most a point.
  The lines in the ambient space not covered by any plane in $\mathcal{C}_3$ will be called \emph{free}.
  All lines in $\mathcal{C}_2$ and all lines passing through $P$ are free.

  For a point $Q$, let $\mathcal{C}_3(Q)$ be the set of all planes in $\mathcal{C}_3$ passing through $Q$.
  Clearly, $\mathcal{C}_3(P) = \emptyset$.
  For $Q \neq P$, $\#\mathcal{C}(Q) \leq q^4 + q^2$, since otherwise all the points of the ambient space, including $P$, would be covered by some element in $\mathcal{C}_3(Q)$.%
  \footnote{For $\#\mathcal{C}_3(Q) = q^4 + q^2 + 1$, the image of $\mathcal{C}_3(Q)$ modulo $Q$ would be a line spread in $V/Q \cong \F_q^6$.}
  We count the set $X$ of flags $(Q,E)$ with $Q\in\gauss{V}{1}{}$ and $E\in\mathcal{C}_3$ in two ways.
  Since $Q < E$,
  \[
	  \# X = \#\mathcal{C}_3 \cdot \gauss{3}{1}{q} = q^3(q^2+1)(q^3+1)\cdot(q^2 + q + 1)\text{.}
  \]
  On the other hand,
  \[
	  \#X
	  = \sum_{Q\in\gauss{V}{1}{}} \#\mathcal{C}_3(Q)
	  \leq (\gauss{7}{1}{q} - 1)\cdot (q^4 + q^2)
	  = q(q^3 + 1)(q^2 + q + 1)\cdot q^2(q^2 + 1) \text{.}
  \]
  Thus, we have in fact equality, which implies $\#\mathcal{C}_3(Q) = q^4 + q^2$ for all $Q \neq P$.

  Modulo $Q$, the $q^4 + q^2$ planes in $\mathcal{C}_3(Q)$ form a partial line spread in $V/Q$.
  It is known that every such partial spread is extendible to a spread.%
  \footnote{For example, using that its set of $q+1$ uncovered points corresponds to a linear code of (effective) length $q+1$ whose codewords 
  have a weight that is divisible by $q$. Thus, all non-zero codewords have a weight of $q$ and the corresponding point set has to be a line.}
  Therefore, the set of $q+1$ free lines through $Q$ spans a plane $E(Q)$, and all lines in $E(Q)$ passing through $Q$ are free.
  %Moreover, the line $\langle P,Q\rangle$ is contained in $F(Q)$.

  %%
  Let $Q'\in E(Q) \setminus \{P\}$.
  For $Q' \notin \langle P,Q\rangle$, $E(Q)$ contains the distinct free lines $\langle Q',P\rangle$ and $\langle Q',Q\rangle$, implying that $E(Q') = E(Q)$.
  For $Q'\in\langle P,Q\rangle$, we pick an auxiliary point $Q''\in E(Q)$ with $Q''\notin \langle P,Q\rangle = \langle P,Q'\rangle$.
  By applying the previous case twice we get again that $E(Q') = E(Q'') = E(Q)$.
  Thus, the set $S = \{E(Q) \mid Q\in\gauss{V}{1}{} \setminus\{P\}\}$ is of size
  \[
	  \frac{\gauss{7}{1}{q} - 1}{\gauss{3}{1}{q} - 1}
	  = \frac{q^7 - q}{q^3 - q}
	  = q^4 + q^2 + 1\text{.}\footnotemark
  \]
  \footnotetext{The image of $S$ modulo $P$ is a line spread in $V/P \cong \F_q^6$.}
  Every line $L\in\mathcal{C}_2$ is free and therefore contained in a plane $E(Q)$ (any point $Q$ on $L$ does the job).
  Moreover, a plane $E(Q)$ cannot contain more than one line from $\mathcal{C}_2$, as any two lines in a plane intersect.
  Therefore $\#\mathcal{C}_2 \leq \#S = q^4 + q^2 + 1$.
\end{proof}

\begin{remark}
The proof of Lemma~\ref{lemma_excluded_1} shows that the free lines are precisely those contained in the planes $E(Q)$.
Thus, $\mathcal{C}_3 \cup S$ covers each line in $\PG(V)$ exactly once.
Such a set of planes is known as a $q$-analog of the Fano plane.
The question for its existence is open for every single value of $q$ and arguably the most important open problem in the theory of $q$-analogs of designs, see \cite{Braun-Kiermaier-Wassermann-2018-COST1} for a survey.
\end{remark}

\begin{proposition}
  \label{prop_a_q_8_3}
  $A_q(8,3)\le q^{12}+3q^{10}+q^9 +3q^8+3q^7+3q^6+5q^5+3q^4+q^3+4q^2+2q-1$ for $q\ge 3$ and $A_2(8,3)\le 9260$.  
\end{proposition}
\begin{proof}
  Let $\mathcal{C}$ be a subspace code of $\mathbb{F}_q^8$ with minimum subspace distance $d=3$ and $a_i$ denote the number 
  of its codewords of dimension $i$. Due to duality we can assume $\#\mathcal{C}\le 2\cdot (a_0+a_1+a_2+a_3)+a_4$. Of course we have 
  $a_i\le A_q(8,4;i)$ for all $0\le i\le 4$, i.e., $a_0,a_1\le 1$, $a_2\le q^6+q^4+q^2+1$, 
  \begin{eqnarray*}
    a_4&\le&\gaussmnum{8}{3}{q}/\gaussmnum{4}{3}{q}=(q^2-q+1)(q^4+1)\gaussmnum{7}{1}{q}\\
    &=&q^{12}+q^{10}+q^9+2q^8+q^7+2q^6+q^5+2q^4+q^3+q^2+1,
  \end{eqnarray*}
  and  
  \begin{eqnarray*}
    a_3&\le& \left\lfloor \frac{A_q(7,4;2)\cdot \gaussmnum{8}{1}{q}}{\gaussmnum{3}{1}{q}}\right\rfloor
    =\left\lfloor \frac{(q^5+q^3+1)\cdot (q^8-1)}{q^3-1}\right\rfloor\\
    &=&q^{10}+q^8+q^7+2q^5+q^4+q^2+q-1,
  \end{eqnarray*}      
  since $\left\lfloor\frac{q+2}{q^2+q+1}\right\rfloor=0$. 
  
  If $a_0=1$, then $a_1=a_2=0$, so that 
  $$
    \#\mathcal{C}\le q^{12}+3q^{10}+q^9+4q^8+3q^7+2q^6+5q^5+4q^4+q^3+3q^2+2q+1.
  $$
  Thus, we can assume $a_0=0$ in the following and consider the set of codewords containing a point $P$. Modulo $P$ 
  the dimension distribution is given by $0^{b_0}1^{b_1}2^{b_2}3^{b_3}$, where obviously $b_i\le A_q(7,4;i)$ for $0\le i\le 3$, 
  i.e., $b_0,b_1\le 1$, $b_2\le q^5+q^3+1$, and $b_3\le \gaussmnum{7}{2}{q}/\gaussmnum{3}{2}{q}=\gaussmnum{7}{1}{q}\cdot(q^2-q+1)=
  q^8+q^6+q^5+q^4+q^3+q^2+1$. 
  To each possible dimension distribution we assign a score
  $$
    \Gamma_8(b)=\frac{b_3}{\gaussmnum{4}{1}{q}} +2\cdot \sum_{i=0}^2 \frac{b_i}{\gaussmnum{i+1}{1}{q}}.
  $$
  If the score of each dimension distribution that occurs at a point $P$ in $\mathcal{C}$ is upper bounded by $\omega$, then 
  we have $\#\mathcal{C}\le \omega\cdot \gaussmnum{8}{1}{q}$. The score of a vector $b\in\mathbb{N}^4$ is of course at least as large 
  as the score of a vector $b'\in\mathbb{N}^4$ if $b\ge b'$ componentwise, so that we just have to consider the feasible dimension 
  distributions that are maximal with respect to this relation. These are given by
  \begin{enumerate}
    \item[(1)] $0^1 3^{b_3}$, where $b_3=\gaussmnum{7}{1}{q}\cdot(q^2-q+1)$;
    \item[(2)] $1^1 2^{b_2} 3^{b_3}$, where $0\le b_2\le q^5+q^3+1$, and $b_3\le \frac{\left(\gaussmnum{7}{1}{q}-1\right)\cdot(q^4+q^2)}{q^2+q+1}$;
    \item[(3)] $2^{b_2} 3^{b_3}$, where $0\le b_2\le q^5+q^3+1$, and $b_3\le \frac{\gaussmnum{7}{1}{q}(q^4+q^2+1)-(q+1)b_2}{q^2+q+1}$.
  \end{enumerate}  
  If $b_0=1$, then $b_1=b_2=0$, which gives case~(1). For the remaining cases we may again consider the set of codewords modulo a common point, which 
  then live in $\mathbb{F}_q^6$. Here the possible maximal dimension distributions are
  \begin{itemize}
    \item $0^1$;
    \item $1^1 2^{q^4+q^2}$;
    \item $2^{q^4+q^2+1}$.
  \end{itemize}
  If $b_1=0$ we can directly conclude the stated upper bound for $b_3$ in case~(3). If $b_1=1$ we observe that no $q^4+q^2+1$ planes 
  can meet in a common point, cf.~the proof of Lemma~\ref{lemma_excluded_1}, so that we obtain the stated upper bound for $b_3$ in case~(2). 
  Of course we may round down the, eventually fractional, upper bound for $b_3$ to an integer. 
  The scores for the three unrounded cases are given by
  \begin{enumerate}
    \item[(1)] $q^5-q^4+q^3+q+1+\frac{q^2+2}{q^3+q^2+q+1}$;
    \item[(2)] $\frac{2}{q+1}+(q^2-q+1)q^3+\frac{2b_2}{q^2+q+1}$;
    \item[(3)] $\frac{\gaussmnum{7}{1}{q}(q^2-q+1)}{(q+1)(q^2+1)}+\frac{2q^2+1}{(q^2+1)(q^2+q+1)}\cdot b_2$.
  \end{enumerate}     
  In cases (2) and (3) the scores are strictly increasing in $b_2$ (which also remains valid if be round the upper bound for $b_3$ to an 
  integer). Plugging in $b_2=q^5+q^3+1$ gives the following upper bounds for the scores
  \begin{enumerate}
    \item[(1)] $q^5-q^4+q^3+q+1+\frac{q^2+2}{q^3+q^2+q+1}$;
    \item[(2)] $q^5-q^4+3q^3-2q^2+2q+\frac{2q+4}{q^3+2q^2+2q+1}$;
    \item[(3)] $q^5-q^4+3q^3-2q^2+2q+\frac{-q^4+2q^2+2}{q^5+2q^4+3q^3+3q^2+2q+1}$.
  \end{enumerate} 
  So, case (2) gives the largest upper bound for the score for all $q$ so that we would obtain an upper bound for $A_q(8,3)$. 
  For $q=2$ we would obtain $A_2(8,3)\le\left\lfloor 9277.142857\right\rfloor=9277$. However, this 
  bound can be slightly improved. The stated score for case (2) corresponds to a subspace code in $\mathbb{F}_q^7$ with dimension distribution 
  $1^1 2^{q^5+q^3+1} 3^{q^8+q^6+q^3}$, i.e., the code excluded in Lemma~\ref{lemma_excluded_1}. We can easily check that the 
  second best score in case (2) is obtained if a plane is removed. In case (3) we can perform the rounding for the upper bound for $b_3$, 
  which gives $b_3\le q^8+q^6+q^5+q^3+q^2-q+1$ for $b_2=q^5+q^3+1$, since $\left\lfloor\frac{-q-2}{q^2+q+1}\right\rfloor=-1$. (Decreasing 
  $b_2$ instead gives a lower score, even without rounding the corresponding upper bound for $b_3$.) Thus, we obtain the following improved 
  upper bounds for the scores
  \begin{enumerate}
    \item[(1)] $0^1 3^{q^8+q^6+q^5+q^4+q^3+q^2+1}$: $q^5-q^4+q^3+q+1+\frac{q^2+2}{q^3+q^2+q+1}$;
    \item[(2)] $1^1 2^{q^5+q^3+1} 3^{q^8+q^6+q^5+q^3-1}$: $q^5-q^4+3q^3-2q^2+2q+\overset{<1}{\overbrace{\frac{2q^3+3q^2+q+3}{q^5+2q^4+3q^3+3q^2+2q+1}}}$;
    \item[(3)] $2^{q^5+q^3+1} 3^{q^8+q^6+q^5+q^3+q^2-q+1}$: $q^5-q^4+3q^3-2q^2+2q+\frac{-q^4+q^2+3}{q^5+2q^4+3q^3+3q^2+2q+1}$.
  \end{enumerate} 
  Again case (2) obtains the best score, then case (3), and then case (1). For $q=2$, we obtain 
  $A_2(8,3)\le\left\lfloor 9260.142856\right\rfloor=9260$. Since 
  $\left\lfloor \frac{2q+4}{q^2+q+1}\right\rfloor=0$ for $q\ge 3$, we obtain an upper bound of 
  $$
    A_q(8,3)\le q^{12}+3q^{10}+q^9 +3q^8+3q^7+3q^6+5q^5+3q^4+q^3+4q^2+2q-1.
  $$ 
\end{proof}  

We remark that the upper bound of Proposition~\ref{prop_a_q_8_3} can almost surely be decreased by $1$, since the considered 
fractional solution violates $a_3\le A_q(8,4;3)$ by $1-\frac{3}{q^2+q+1}$. However, the corresponding analysis might get quite involved, 
i.e., one has to solve an ILP.

Given the proof of Lemma~\ref{lemma_excluded_1}, it seems more reasonable to also exclude cases where the number of planes 
is strictly less than $q^8+q^6+q^5+q^3$. For $q=2$ the later number is $360$ while the largest known number of planes in $\mathbb{F}_2^7$ 
with subspace distance $d=3$ is $333$ \cite{heinlein2017subspace}. We remark that the exclusion of dimension distribution $1^1 2^{q^5 + q^3} 3^{q^8+q^6+q^5+q^3}$ 
in $\mathbb{F}_q^7$ for subspace distance $d=3$ would have been sufficient for the conclusion in the proof of Proposition\ref{prop_a_q_8_3}. 
However, we think that Lemma~\ref{lemma_excluded_1} is interesting in its own right and the presented tightening does not complicate the proof. 

\section{Conclusion}
\label{sec_conclusion}

We have generalized the underlying idea of the Johnson bound for constant dimension codes to mixed dimension subspace codes. 
As in the case of the Etzion-Vardy ILP we also have to deal with integer linear programs. However, things get more complicated. 
Nevertheless parametric improved upper bounds for $A_q(v,v-4)$ and $A_q(8,3)$ have been obtained. We illustrate our results 
with a small table of improvements for the binary case and small parameters:

\begin{center}
  \begin{tabular}{lllllll}
    \hline
      parameters & improved cdc & ILP E/V & SDP & \textbf{johnson} & details & bklb \\
    \hline
    %%$A_2(8,3)$ & 9633 & 9365 & 9191 & 9260 & Proposition~\ref{prop_a_q_8_3} & 5687 \\
    $A_2(10,5)$ & 48394 & 48336 & 49394 & 48104 & Lemma~\ref{lemma_a_2_10_5} & 32940 \\
    $A_2(10,6)$ & 48394 & \textit{48336} & - & 38275 & Lemma~\ref{lemma_n_minus_4_even} & 32890 \\ 
    $A_2(11,7)$ & 8844 & 9120$^\star$ & 8990 & 8842 & Lemma~\ref{lemma_n_minus_4_odd} & 8067 \\
    $A_2(13,9)$ & 34058 & 34591$^\star$ & 34306 & 34056 & Lemma~\ref{lemma_n_minus_4_odd} & 32514 \\
    \hline
  \end{tabular}  
\end{center}

Here {\lq\lq}improved cdc{\rq\rq} refers to Lemma~\ref{lemma_improved_cdc_upper_bound}, {\lq\lq}ILP E/V{\rq\rq} to the 
ILP of Etzion and Vardy, see Section~\ref{section_known_bounds}, {\lq\lq}SDP{\rq\rq} to results based on semidefinite programming, 
see \cite{MR3063504}, {\lq\lq}johnson{\rq\rq} to the results obtained in this paper, and {\lq\lq}bklb{\rq\rq} to the currently best known 
lower bound, see \cite{TableSubspacecodes}. If the entry of {\lq\lq}ILP E/V{\rq\rq} is written in italics, then the value for subspace 
distance $d-1$ is taken. If the entry 
is marked with $^\star$ then the value of \cite{MR3063504} is taken.

\section*{Acknowledgement}
The third author was supported in part by the grant KU 2430/3-1 -- Integer Linear Programming Models for 
Subspace Codes and Finite Geometry -- from the German Research Foundation.

%\bibliography{johnson_mdc}

\begin{thebibliography}{10}

\bibitem{bachoc2012invariant}
C.~Bachoc, D.~C. Gijswijt, A.~Schrijver, and F.~Vallentin.
\newblock Invariant semidefinite programs.
\newblock In {\em Handbook on semidefinite, conic and polynomial optimization},
  pages 219--269. Springer, 2012.

\bibitem{MR3063504}
C.~Bachoc, A.~Passuello, and F.~Vallentin.
\newblock Bounds for projective codes from semidefinite programming.
\newblock {\em Advances in Mathematics of Communications}, 7(2):127--145, 2013.

\bibitem{beutelspacher1975partial}
A.~Beutelspacher.
\newblock Partial spreads in finite projective spaces and partial designs.
\newblock {\em Mathematische Zeitschrift}, 145(3):211--229, 1975.

\bibitem{Braun-Kiermaier-Wassermann-2018-COST1}
M.~Braun, M.~Kiermaier, and A.~Wassermann
\newblock $q$-analogs of designs: Subspace designs.
\newblock In {\em Network Coding and Subspace Designs} (M. Greferath, M. O. Pav\v{c}evi\'{c}, N. Silberstein, and M. \'{A}. V\'{a}zquez-Castro, eds.), Signals and Communication Theory, Springer, Cham,
  2018, pages~171--211.

\bibitem{nets_and_spreads}
D.~Drake and J.~Freeman.
\newblock Partial $t$-spreads and group constructible $(s,r,\mu)$-nets.
\newblock {\em Journal of Geometry}, 13(2):210--216, 1979.

\bibitem{etzionsurvey}
T.~Etzion and L.~Storme.
\newblock Galois geometries and coding theory.
\newblock {\em Designs, Codes and Cryptography}, 78(1):311--350, 2016.

\bibitem{MR2810308}
T.~Etzion and A.~Vardy.
\newblock Error-correcting codes in projective space.
\newblock {\em IEEE Transactions on Information Theory}, 57(2):1165--1173,
  2011.

\bibitem{COSTbook}
M.~Greferath, M.~Pav{\v{c}}evi{\'c}, N.~Silberstein, and M. {\'A}.~V{\`a}zquez-Castro,
  editors.
\newblock {\em {N}etwork {C}oding and {S}ubspace {D}esigns}.
\newblock Springer, 2017.

\bibitem{paper257}
D.~Heinlein, T.~Honold, M.~Kiermaier, S.~Kurz, and A.~Wassermann.
\newblock Classification of optimal binary subspace codes of length $8$,
  constant dimension $4$ and minimum distance $6$.
\newblock {Designs, Codes and Cryptography}, to appear, {\em arXiv preprint 1711.06624}, 2017.

\bibitem{heinlein2018new}
D.~Heinlein and F.~Ihringer-
\newblock New and updated semidefinite programming bounds for subspace codes.
\newblock {\em arXiv preprint 1809.09352}, 2018.

\bibitem{heinlein2017subspace}
D.~Heinlein, M.~Kiermaier, S.~Kurz, and A.~Wassermann.
\newblock A subspace code of size $333 $ in the setting of a binary $q$-analog of the Fano plane.
\newblock {\em arXiv preprint 1708.06224}, 2017.

\bibitem{TableSubspacecodes}
D.~Heinlein, M.~Kiermaier, S.~Kurz, and A.~Wassermann.
\newblock Tables of subspace codes.
\newblock {\em arXiv preprint 1601.02864}, 2016.

\bibitem{heinlein2017asymptotic}
D.~Heinlein and S.~Kurz.
\newblock Asymptotic bounds for the sizes of constant dimension codes and an
  improved lower bound.
\newblock In {\em 5th International Castle Meeting on Coding Theory and
  Applications}, pages 1--30, 2017.
\newblock arXiv preprint 1705.03835.

\bibitem{heinlein2017new}
D.~Heinlein and S.~Kurz.
\newblock An upper bound for binary subspace codes of length 8, constant
  dimension 4 and minimum distance 6.
\newblock In {\em The Tenth International Workshop on Coding and Cryptography},
  pages 1--9, 2017.
\newblock arXiv preprint 1703.08712.

\bibitem{heinlein2018binary}
D.~Heinlein and S.~Kurz.
\newblock Binary subspace codes in small ambient spaces.
\newblock {\em Advances in Mathematics of Communications},  13(4):817--839, 2018.

\bibitem{MR3329980}
T.~Honold, M.~Kiermaier, and S.~Kurz.
\newblock Optimal binary subspace codes of length $6$, constant dimension $3$
  and minimum subspace distance $4$.
\newblock In {\em Topics in finite fields}, volume 632 of {\em Contemp. Math.},
  pages 157--176. Amer. Math. Soc., Providence, RI, 2015.

\bibitem{honold2016classification}
T.~Honold, M.~Kiermaier, and S.~Kurz.
\newblock Classification of large partial plane spreads in ${P}{G}(6,2)$ and
  related combinatorial objects.
\newblock {\em Journal of Geometry}, doi: 10.1007/s00022-018-0459-6, to appear, {\em arXiv preprint 1606.07655}, 2016.

\bibitem{honold2015constructions}
T.~Honold, M.~Kiermaier, and S.~Kurz.
\newblock Constructions and bounds for mixed-dimension subspace codes.
\newblock {\em Advances in Mathematics of Communication}, 10(3):649--682, 2016.

\bibitem{partial_spreads_and_vectorspace_partitions}
T.~Honold, M.~Kiermaier, and S.~Kurz.
\newblock Partial spreads and vector space partitions.
\newblock In {\em Network Coding and Subspace Designs} (M. Greferath, M. O. Pav\v{c}evi\'{c}, N. Silberstein, and M. \'{A}. V\'{a}zquez-Castro, eds.), Signals and Communication Theory, Springer, Cham,
  2018, pages 131--170.

\bibitem{johnson1962new}
S.~Johnson.
\newblock A new upper bound for error-correcting codes.
\newblock {\em IRE Transactions on Information Theory}, 8(3):203--207, 1962.

\bibitem{kiermaier2017improvement}
M.~Kiermaier and S.~Kurz.
\newblock An improvement of the {J}ohnson bound for subspace codes.
\newblock {\em arXiv preprint 1707.00650}, 2017.

\bibitem{kurz2016improved}
S.~Kurz.
\newblock Improved upper bounds for partial spreads.
\newblock {\em Designs, Codes and Cryptography}, 85(1):97--106, 2017.

\bibitem{kurz2017packing}
S.~Kurz.
\newblock Packing vector spaces into vector spaces.
\newblock {\em The Australasian Journal of Combinatorics}, 68(1):122--130,
  2017.

\bibitem{MR0465510}
F.~J. MacWilliams and N.~J.~A. Sloane.
\newblock {\em The theory of error-correcting codes. {II}}.
\newblock {N}orth-Holland Publishing Co., Amsterdam-New York-Oxford, 1977.
\newblock {N}orth-Holland Mathematical Library, Vol. 16.

\bibitem{nastase2016maximum}
E.~N{\u{a}}stase and P.~Sissokho.
\newblock The maximum size of a partial spread in a finite projective space.
\newblock {\em Journal of Combinatorial Theory, Series A}, 152:353--362, 2017.

\bibitem{segre1964teoria}
B.~Segre.
\newblock Teoria di galois, fibrazioni proiettive e geometrie non
  desarguesiane.
\newblock {\em Annali di Matematica Pura ed Applicata}, 64(1):1--76, 1964.

\bibitem{tonchev1998codes}
V.~D. Tonchev.
\newblock Codes and designs.
\newblock {\em Handbook of coding theory}, 2:1229--1267, 1998.

\bibitem{xia2009johnson}
S.-T. Xia and F.-W. Fu.
\newblock Johnson type bounds on constant dimension codes.
\newblock {\em Designs, Codes and Cryptography}, 50(2):163--172, 2009.

\end{thebibliography}
%\bibdata{johnson_mdc}
%\bibliographystyle{amsplain}
%\bibliographystyle{abbrv}

\end{document}